\documentclass[12pt,reqno]{amsart}
\usepackage[latin1]{inputenc}
\usepackage{hyperref,color}
\newtheorem{lemma}{Lemma}[section]

\newtheorem{theorem}[lemma]{Theorem}
\newtheorem{corollary}[lemma]{Corollary}
\newtheorem{remark}[lemma]{Remark}

\newtheorem{example}[lemma]{Example}
\newtheorem{question}[lemma]{Question}

\numberwithin{equation}{section}

\DeclareMathOperator{\dist}{dist}

\newcommand{\C}{\mathbb{C}}

\newcommand{\N}{\mathbb{N}}
\newcommand{\R}{\mathbb{R}}

\newcommand{\T}{\mathbb{T}}
\newcommand{\Z}{\mathbb{Z}}
\newcommand{\D}{\mathbb{D}}

\newcommand{\cC}{\mathcal{C}}
\newcommand{\cA}{\mathcal{A}}

\begin{document}
\title[Toeplitz and Hankel operators]
{Toeplitz and Hankel operators between distinct Hardy spaces}

\author[Le\'snik]{Karol Le\'snik}
\address[Le\'snik]{%
Institute of Mathematics,
Pozna\'n University of Technology, 
ul. Piotrowo 3a, 
60-965 Pozna\'{n}, Poland}
\email{klesnik@vp.pl}

\begin{abstract}
The paper presents the background for Toeplitz and Hankel  operators acting between distinct Hardy type spaces 
over the unit circle $\T$. We characterize possible symbols of such operators and prove general versions of Brown-Halmos and Nehari theorems. The lower bound for measure of noncomactness of Toeplitz operator is also found. Our approach allows Hardy spaces associated with arbitrary rearrangement invariant spaces, but a main part of results is new even for the classical case of $H^p$ spaces. 
\end{abstract}
\keywords{%
Toeplitz operator, 
Hankel operator, 
rearrangement-invariant  spaces, 
Hardy spaces}
\subjclass[2010]{47B35, 46E30, 30H10}
\vspace{-5 mm}

\maketitle

\rightline{\it Dedicated to the memory of}
\rightline{\it Professor Pawe\l\ Doma\'nski}
\rightline{\it (1959--2016)}
\vspace{5 mm}


\section{Introduction}
Classical Toeplitz $T_a$ and Hankel $H_a$ operators on Hardy space $H^2$ (on the unit circle $\T$) are defined by 
\begin{equation}\label{TH}
T_a\colon f\mapsto P(af)\ {\rm\ and\ } H_a\colon f\mapsto P(aJf),
\end{equation}
where $P$ is the Riesz projection, $J$ is the flip operator and the function $a\in L^{\infty}$ is called the symbol of $T_a$ and $H_a$, respectively. 

Theory of Toeplitz and Hankel operators acting on $H^p$ spaces, as well as on a number of another function spaces is very well developed and still widely investigated. Moreover, such operators are interesting not only from the point of view of operator theory, but they are intimately connected with harmonic analysis, prediction theory and approximation theory (see for example \cite{Pel03}).
However, in the literature Toeplitz and Hankel operators are mainly considered to act from one to the same space.

Suppose now we leave the above definition (\ref{TH}) unchanged, but take a symbol $a\in L^r$ for some $1<r<\infty$. In such a case $T_a$ and $H_a$ need not be bounded on any $H^p$ space, but they act boundedly from $H^p$ to $H^q$ if $1<q<p<\infty$ and $\frac{1}{p}+\frac{1}{r}=\frac{1}{q}$. It appears that almost nothing is known about such operators. Among a huge number of papers considering Toeplitz and Hankel operators we were able to find only few, where they act between distinct spaces. This number includes  papers of Tolokonnikov \cite{Tol87} and of Tolokonnikov and Volberg \cite{TV87}. In the first the symbols  of Toeplitz and Hankel operators acting between distinct $H^p$ spaces were determined, while the second is devoted to approximation problem connected with the representation of Hankel operators considered between abstract Hardy type spaces. Except these two papers one can find investigations of Toeplitz and Hankel operators acting from some Hardy type space into $H^1$ in the Janson, Peetre and Semmes paper \cite{JPS84} and a generalization of these investigations for Hardy spaces over more complicated domains in \cite{BG10}. 

The goal of this paper is to present an unified background for Toeplitz and Hankel operators acting between distinct Hardy spaces, i.e. $T_a,H_a\colon H[X]\to H[Y]$, where $X,Y$ are rearrangement invariant spaces. The main results are general versions of Brown--Halmos and Nehari theorems. In such a general situation symbols $a$ belong to the space of pointwise multipliers $M(X,Y)$. In consequence, a deeper theory of function spaces, pointwise multipliers, pointwise products and factorization comes into play.  

The paper is organized as follows. In the second section  we collect required definitions and notation, that will be used through the paper. The third section contains a number of technical results describing basic properties of Hardy type spaces built upon rearrangement invariant function spaces on the unit circle $\T$. 

The fourth section is devoted to Toeplitz operators. In the classical case of $H^2$ the following theorem characterizes bounded Toeplitz operators.
\newline
{\bf Brown--Halmos theorem \cite{BH63/64}} {\it
A bounded linear operator $A\colon H^2\to H^2$ satisfies
\begin{equation}\label{Toep0}
\langle A\chi_j,\chi_k\rangle=a_{k-j}
\end{equation}
for some sequence 
$(a_k)_{k\in \mathbb{Z}}$ and all $j,k\geq 0$ (where $(\chi_n)_{n=0}^{\infty}$ is a standard basis of $H^2$) if and only if there exists a unique $a\in L^{\infty}$ such that $A=T_a$, i.e. $A\colon f\mapsto P(af)$ and  $\widehat{a}(n)=a_n$ for all $n\in\Z$.  Moreover,
\begin{equation}\label{Toep0norm}
\|a\|_{L^{\infty}}=\|T_a\|_{H^2\to H^2}.
\end{equation}
}\newline 
It not only identifies possible symbols of bounded  Toeplitz operators on $H^2$, but mainly says  that each operator satisfying (\ref{Toep0}), i.e. having Toeplitz matrix with respect to the standard basis of $H^2$, has the representation of the form $T_a$, where $a\in L^{\infty}$ is uniquely determined. We give an analogue of the  Brown--Halmos theorem for the case of operators  acting from $H[X]$ to $H[Y]$, under some mild assumptions on spaces $X,Y$. The result seems to be new even in the case of $T_a\colon H^p\to H^q$. Let us mention also, that the version of Brown--Halmos theorem  for the case $X=Y$ has been already proved  in \cite{K04}, but even in this particular case our assumptions are less restrictive. Moreover, we discuss also the case of nonseparable spaces $X$ and $Y$. 

In the main, fifth section, Hankel operators are taken into account. While the previous section is rather analogous to the classical case, except some technicalities, situation for Hankel operators makes much more interesting. Let us recall the statement of the classical Nehari theorem. \newline
{\bf Nehari theorem \cite{Neh57}} {\it 
A continuous linear operator $A:H^2\to H^2$ satisfies 
\begin{equation}\label{Hankel condition0}
\langle A\chi_j,\chi_k\rangle=a_{k+j+1} 
\end{equation}
for some sequence 
$(a_k)_{k>0}$ and all $j,k\geq 0$ if and only if there exists $a\in L^{\infty}$ (not unique) such that $\hat a(n)=a_n$ for each
$n>0$ and $A=H_a$, i.e. $A:f\mapsto PaJf$. Moreover, 
\[
\|H_a\|_{H^2\to H^2}=\dist_{L^{\infty}}(a,\overline{H^{\infty}}).
\]}\newline 
Thus, the theorem characterizes operators with Hankel matrices and their symbols.  However, we point out that, in contrast to Brown--Halmos theorem, a symbol $a$ is not unique (i.e. the operator remains the same if $a$ is modified by adding arbitrary function $b$ satisfying $Pb=0$, since only Fourier coefficients of $a$ for $n>0$ appears in (\ref{Hankel condition0})). 
In this section we prove a general Nehari theorem for Hankel operators acting from $H[X]$ to $H[Y]$, under some assumptions on spaces $X,Y$. 
Let us mention, that modern proofs of Nehari theorem base on the (strong) factorization of $H^1$ function $f$ into product $f=gh$, where $g,h\in H^2$ (see for example \cite[Theorem 2.11]{BS06} or \cite[Theorem 1.1]{Pel03}). A direct translation of this idea together with the Lozanovskii factorization theorem for Hardy spaces (i.e. $H[X]\odot H[X']=H^1$, where $X'$ is the K\"othe dual of $X$) was used in \cite{K04} to prove the Nehari theorem for $H_a:H[X]\to H[X]$ (see also \cite{Ha98}, where the same subject was undertaken). Of course, the generalized Lozanovskii-like factorization would do the job also in our setting, however the assumption that $X$ factorizes $Y$ (i.e. $X\odot M(X,Y)=Y$) is rather restrictive (see \cite{KLM14} for extensive studies of this problem) and we expect weaker assumptions for the general Nehari theorem. On the other hand, as it was noticed by Coifman, Rochberg and Weiss \cite{CRW76} (see also \cite{JPS84} and \cite{TV87}) the strong factorization may be  replaced by the weak one (i.e. $f=\sum_kg_kh_k$ instead of $f=gh$). 
However, theory of such factorization is not very well developed and it is not at all applicable in a general setting (the space of symbols of Hankel operators were described in terms of weak factorization in \cite{TV87}, but it appeared that the authors were able to give concrete representation  only in cases when strong factorization holds). 
Therefore, instead of weak factorization, we base our proof of general Nehari theorem on the concept of Banach envelopes, which works pretty well in this setting and, indeed,  gives a weak factorization, as a byproduct  (see discussion after Lemma \ref{Ban-env-Cor}). This section is finished by an extensive discussion on assumptions of the main theorem and we give some examples for concrete types of spaces, like Orlicz and Lorentz spaces. 

We finish the paper estimating the measure of noncompactness of Toeplitz operator $T_a$ in terms of Fourier coefficients of its symbol $a$.

\section{Notions and notations}

Let $\T$ be the unit circle equipped with the normalized Lebesgue measure 
$dm(t)=|dt|/(2\pi)$. Let $L^0:=L^0(\T,m)$ be the space of all 
measurable complex-valued almost everywhere finite functions on $\T$. 
As usual, we do not distinguish functions, which are equal almost 
everywhere (for the latter we use the standard abbreviation a.e.). 
The characteristic function 
of a measurable set $E\subset\T$ is denoted by $\chi_E$. 

A complex quasi-Banach space $X\subset L^0(\T,m)$ is called a
quasi-Banach function space (q-B.f.s for short) if\\
(a)  $f\in X$, $g\in L^0$ and $|g| \le |f|$  a.e. $\Rightarrow \ g\in X$ and $\|g\|_X \le \|f\|_X$
(the ideal property),\\
(b) $L^{\infty}\subset X$, \\
(c) $X$ has the semi-Fatou property, i.e.  $(f_n)\subset X$, $f\in X$ and
$0\le f_n\uparrow f$ a.e. implies $\|f\|_X = \sup_{n\in\N}\|f_n\|_X$.

If $X$ as above is a Banach space, we will call it the Banach function space (B.f.s. for short). 

Spaces as above are also called K\"othe spaces (see for example \cite[page 28]{LT79}, \cite[Chapter 15]{Za67}), while the name B.f.s. comes from \cite[Chapter 1]{BS88}, \cite[p. 114]{MN91} or \cite[p. 161]{Mal89} (cf. \cite[pages 40--43]{KPS82}). Notice however, that our definition differs a little from respective definitions in the mentioned books. In fact, assumptions in our definition are weaker than in \cite{BS88}, but stronger than in \cite{LT79, Za67, KPS82, Mal89, MN91}. The point (a) is crucial in all of them. Point (b) is satisfied by each rearrangement invariant spaces and we will focus only on such spaces, thus we assumed it already for B.f. spaces to simplify presentation. Finally, semi-Fatou property from the point (c) will be crucial in few places, but we cannot replace it by the stronger Fatou property (which is assumed for B.f. spaces in \cite{BS88}), because we will work a lot with subspaces of order continuous elements, which, in general, need not satisfy the Fatou property. Finally, notice that classical spaces, such as Lebesgue, Orlicz and Lorentz spaces, fulfill conditions of our definition. 

It is known, that for q-B.f. spaces $X,Y$ inclusions $X\subset Y$ are always continuous, i.e. there is $C>0$ such that $\|f\|_Y\leq C\|f\|_X$ for each $f\in X$. It follows, for example, from continuity of embedding $X,Y\subset L^0$ (see for example \cite[Proposition 2.7.2]{Ro85}) and the closed graph theorem (see \cite[pages 9--11]{KPR84} for discussion on classical theorems in quasi-Banach case). 
We will write $X=Y$ if $X$ and $Y$ coincide as sets 
and there are positive constants $c_1,c_2$ such that  
$c_1\|f\|_X\le \|f\|_Y\le c_2\|f\|_X$ for all $f\in X$ (the latter inequalities will be also denoted as $\|f\|_Y\approx \|f\|_X$), and 
$X\equiv Y$ if $c_1=c_2=1$. 

A q-B.f.s. $X$ has the Fatou property ($X\in (FP)$ for short) 
when given a sequence $(f_n)_{n\in\N}\subset X$ and $f\in L^0$ satisfying
$0\le f_n\uparrow f$ a.e. as $n\to\infty$ and 
$\sup_{n\in\N}\|f_n\|_X<\infty$ there holds $f\in X$ and 
$\|f\|_X = \sup_{n\in\N}\|f_n\|_X$. 

Recall that $f\in X$ is said to be an order continuous element, 
if for each $(f_n)_{n\in\N}\subset X$ satisfying
$0\le f_n\le |f|$  for all $n\in\N$ and $f_n\to 0$ a.e. as $n\to\infty$, there holds $\|f_n\|_X\to 0$ 
as $n\to\infty$. The subspace of order continuous elements of $X$ 
is denoted by $X_o$. Evidently, $X_o$ enjoys the semi-Fatou property. 
We say that $X$ is order continuous, when $X=X_o$, which is equivalent with separability of $X$. 

For a q-B.f.s. $X$, its associate space (K\"othe dual) 
$X'$ is defined as the space of functions $g\in L^0$ satisfying
\[
\|g\|_{X'}:=\sup\left\{
\int_\T |f(t)g(t)|dm(t)\colon \|f\|_X \le 1
\right\}<\infty.
\]
Notice that $X'$ is nontrivial and $X'\in (FP)$ for each B.f.s. $X$. However, $X'$ may be trivial, i.e. $X'=\{0\}$, when $X$ is just a q-B.f.s.. 
For example, $(L^p)'=\{0\}$ when $0<p<1$. It is  known that a B.f.s. $X$ has the Fatou property if and only if $X''\equiv X$ (see \cite[p. 30]{LT79}). 
Moreover, when $X$ is a B.f.s., the property (c) of definition implies that
\begin{equation}\label{eq:pol-norm-0}
\|f\|_X=\sup\{|\langle f,g\rangle| : g\in X',\ \|g\|_{X'}\le 1\},
\end{equation}
i.e. $\|f\|_X=\|f\|_{X''}$ for each $f\in X$ (see \cite[Proposition 1.b.18]{LT79}). 
Finally, a B.f.s. $X$ satisfies 
\begin{equation}\label{drugi dual}
(X_o)'\equiv X',
\end{equation}
provided  $L^{\infty}\subset (X')_o$. It follows directly from definitions of norms in $X',(X_o)'$  and  the Lebesgue dominated convergence theorem, since the assumption  $L^{\infty}\subset X_o$ implies that simple functions are in $X_o$ and each function from $X$ is a pointwise (a.e.) limit of an increasing sequence of simple functions.


The distribution function $\mu_f$ of $f\in L^0$ is given by
\[
\mu_f(\lambda)=m\{t\in\T\colon |f(t)|>\lambda\},\quad\lambda\ge 0.
\]
Two functions $f,g\in L^0$ are said to be equimeasurable
if $\mu_f(\lambda)=\mu_g(\lambda)$ for all $\lambda\ge 0$. The non-increasing
rearrangement $f^*$ of  $f\in L^0$ is defined by 
\[
f^*(x)=\inf\{\lambda \colon \mu_f(\lambda)\le x\},\quad x\ge 0.
\]

A q-B.f.s. $X$  is called
re\-ar\-range\-ment-invariant (r.i. for short) if for every pair of equimeasurable functions 
$f,g \in L^0$, $f\in X$ implies that $g\in X$ and
$\|f\|_X=\|g\|_X$. 
Lebesgue, Orlicz and Lorentz spaces are examples of r.i. q-B.f. spaces. In general, each r.i. B.f.s. $X$ satisfies inclusion $X\subset L^1$. Moreover, if $X$ is r.i. B.f.s on $\T$ and $X\not =L^{\infty}$, then $L^{\infty}\subset X_o$. We refer to \cite{KPS82} and \cite{BS88} for more informations on non-increasing rearrangements and r.i. spaces.

Let $X$ be a r.i. q-B.f. space. For each $s\in\R_+$ the dilation operator $D_s$ is defined as
\[
(D_s f)(e^{i\theta})=
\left\{
\begin{array}{ll}
f(e^{i\theta s}), & \theta s\in[0,2\pi),\\
0,     &\theta s\not\in[0,2\pi),
\end{array}
\right.
\quad \theta \in[0,2\pi).
\]
It is known (see, for example, \cite{KPS82}) that $D_s$ is bounded on $X$ for each $s>0$ and  limits
\[
\alpha_X=\lim_{s\to 0}\frac{\log \|D_{1/s}\|_{X\to X}}{\log s},
\quad
\beta_X=\lim_{s\to \infty}\frac{\log \|D_{1/s}\|_{X\to X}}{\log s}
\]
exist. The numbers $\alpha_X$ and $\beta_X$ are called 
 lower and upper Boyd indices of $X$, respectively. For an arbitrary
r.i. B.f.s. $X$, its Boyd indices belong to $[0,1]$ and $\alpha_X\le\beta_X$. Moreover, 
\begin{equation}\label{indeksysuma}
\alpha_X+\beta_{X'}=1,
\end{equation}
when $X$ is a B.f.s., thanks to the semi-Fatou property of $X$ (see \cite[Theorem 4.11, p. 106]{KPS82}). 
We say that  Boyd indices are nontrivial if $\alpha_X,\beta_X\in(0,1)$.
More informations on Boyd indices of r.i. B.f. spaces may be found in \cite{BS88, KPS82, LT79}, while the quasi-Banach case was considered in \cite{MS96,Di15}. 

For two B.f. spaces $X$ and $Y$, let $M(X,Y)$ denote the space of pointwise multipliers from $X$ to 
$Y$ defined by 
\[
M(X,Y)=\{f\in L^0 \colon fg\in Y\text{ for all } g\in X\},
\]
equipped with the natural operator norm 
\[
\|f\|_{M(X,Y)}=\sup_{\|g\|_X\le 1}\|fg\|_Y.
\]
It is known (\cite[Theorem~2.2]{KLM13}) that $M(X,Y)$ is r.i. space when $X,Y$ are so. 
Note that it may happen that $M(X,Y)$ contains only the zero function.
For instance, if $1\le p<q<\infty$, then $M(L^p,L^q)=\{0\}$. In general, for two r.i. B.f.s. $X,Y$, $M(X,Y)\not =\{0\}$ if and only if $X\subset Y$. 
On the other hand, if $1\le q\le p\le\infty$ and $1/r=1/q-1/p$, then 
$M(L^p,L^q)\equiv L^r$. Also  $M(X,X)\equiv L^\infty$ for arbitrary B.f.s. $X$. 
We will need one more easy fact about space $M(X,Y)$, for which we cannot give any reference, thus let us state it and prove.

\begin{lemma}\label{order multipliers}
Let $X,Y$ be B.f. spaces such that $X\subset Y$, $L^{\infty}\subset X_o$ and $Y\in (FP)$. Then  $M(X_o,Y)\equiv M(X,Y)$.
\end{lemma}
\proof
It is known (see \cite[(vii) on p. 879]{KLM13}) that $M(Z,Y)\equiv M(Y',Z')$, when $Y\in (FP)$. Applying this together with (\ref{drugi dual}) we get 
\[
M(X_o,Y)\equiv M(Y',(X_o)')\equiv M(Y',X')\equiv M(X'',Y'')\equiv M(X'',Y).
\]
On the other hand, the following inclusions always hold $X_o\subset X\subset X''$, thus
\[
\|f\|_{M(X_o,Y)}\leq \|f\|_{M(X,Y)}\leq \|f\|_{M(X'',Y)},
\]
which means that $M(X_o,Y)\equiv M(X,Y)\equiv M(X'',Y)$. 
\endproof

The space $M(X,Y)$ may be regarded as division of $Y$ by $X$. In virtue of this point of view we define an opposite construction, that is the pointwise product space. 
Given two q-B.f. spaces $X$ and $Y$, the pointwise product $X\odot Y$ is defined by
\begin{equation}\label{product}
X\odot Y=\{gh\colon g\in X,\ h\in Y\}
\end{equation}
equipped with the functional $\|\cdot\|_{X\odot Y}$ 
\[
\|f\|_{X\odot Y}=\inf\{\|g\|_X\|h\|_Y\colon f=gh,\ g\in X,\ h\in Y\}.
\]
It follows from the ideal property of B.f. spaces that $X\odot Y$ is a linear space (products of sequence spaces without the ideal property have been investigated in \cite{Bu87, BG87}).
Given two B.f. spaces $X,Y$ we say that $X$ factorizes $Y$ when $X\odot M(X,Y)=Y$ (factorization of function spaces is widely discussed in \cite{KLM14}). Using this notion, the classical Lozanovskii factorization theorem reads as follows
\begin{equation}\label{Loz fact}
X\odot X'\equiv L^1,
\end{equation}
where $X$ is a  B.f.s. (\cite[Theorem 6]{Lo69}, cf. \cite[Proposition 6]{Re88}).

The Calder\'on--Lozanovskii construction $X^{1-\theta}Y^{\theta}$ is defined for $0<\theta <1$ and  two q-B.f. spaces $X,Y$  by 
\[
X^{1-\theta}Y^{\theta}=\{f\in L^0:|f|=g^{1-\theta}h^{\theta},\ g\in X,\ h\in Y\},
\]
with the (quasi) norm given by 
\begin{equation}\label{CLnorm}
\|f\|_{X^{1-\theta}Y^{\theta}}=\inf\{\max\{\|g\|_X,\|h\|_Y\}:|f|=g^{1-\theta}h^{\theta},\ g\in X,\ h\in Y\},
\end{equation}
(for more informations see \cite{Lo69, Re88}, \cite[Chapter IV]{KPS82}, \cite[Chapter 15]{Mal89}). The second Lozanovskii's theorem that we shall need is the duality theorem, which states that
\begin{equation}\label{lozan dual}
[X^{1-\theta}Y^{\theta}]'\equiv (X')^{1-\theta}(Y')^{\theta}, 
\end{equation}
for B.f. spaces $X,Y$ (\cite[Theorem 1]{Lo69}, cf. \cite[Theorem 1]{Re88}).

For a q-B.f.s. $X$ and $p>1$ one defines its $p$-convexification ($p$-concavication when $0<p<1$) $X^{(p)}$ as 
\[
X^{(p)}=\{f\in L^0\colon |f|^p\in X\}
\]
with the quasi-norm given by $\|f\|_{X^{(p)}}=\||f|^p\|_{X}^{1/p}$ (see \cite[pages 40--59]{LT79}). 
The product space $X\odot Y$ is intimately related with Calder\'on--Lozanovskii construction. In fact, $X\odot Y$ may be represented in the following way
\begin{equation}\label{repres}
X\odot Y\equiv (X^{1/2}Y^{1/2})^{(1/2)},
\end{equation}
(see \cite[Theorem 1 (iv)]{KLM14}). In particular, it explains that $X\odot Y$ is a quasi-Banach space. We will use (\ref{repres}) few times in the sequel, since it allows us to apply the known theory of Calder\'on--Lozanovskii construction to product spaces. 

For $n\in\Z$ and $t\in\T$, let 
$\chi_n(t):=t^n$. The Fourier coefficients of a function $f\in L^1$
are given by
\[
\widehat{f}(n)=\langle f,\chi_n\rangle, \quad n\in\Z,
\]
where
\[
\langle f,g\rangle=
\int_\T f(t)\overline{g(t)}\,dm(t).
\]
Let further $\mathcal{P}=\{\sum_{i=-n}^n\alpha_i\chi_i: \alpha_i\in\C,\ n\geq 0\}$ and $\mathcal{P}_A=\{\sum_{i=0}^n\alpha_i\chi_i :  \alpha_i\in\C,\ n\geq 0\}$ denote the sets of all trigonometric polynomials and all analytic trigonometric polynomials, respectively.

The Riesz projection $P$  is defined for  $f\in L^1$, as 
\[
P\colon f\mapsto \frac{1}{2}(f+i\tilde f+\hat f(0)),
\]
where $\tilde f$ is the conjugate function of $f$ (see \cite[Chapter III]{Kat76} or \cite[Chapter III]{Gar06} for precise definitions). For each $f\in L^p$, $p>1$ equivalent definition of $P$, i.e. 
\[
P\colon \sum_{n=-\infty}^{\infty}\widehat f(n)t^n\mapsto \sum_{n=0}^{\infty}\widehat f(n)t^n,
\]
is meaningful. 
It is known that $P$ is bounded on r.i. q-B.f.s. $X$  if and only if $X$ has nontrivial Boyd indices (in the case of Banach spaces  it follows directly from the Boyd theorem \cite{LT79,KPS82}, while the quasi-Banach case was considered in \cite{Di15} and, before, in \cite{MS96} with an additional assumption of the Fatou property). In the paper $P$ will always stand  for the Riesz projection.

Let $X$ be a r.i. q-B.f.s. such that $X\subset L^1$ (inclusion $X\subset L^1$ holds for each r.i. B.f.s.).  The Hardy space $H[X]$ is defined by
\[
H[X]=\big\{f\in X\colon \widehat{f}(n)=0\quad\mbox{for all}\quad n<0\big\},
\]
with the quasi-norm inherited from $X$ (see for example \cite{Xu92}, \cite{MM09} or \cite{MRP15}, where this kind of Hardy spaces is considered).
Let us mention that Hardy spaces may be equivalently regarded as spaces of analytic functions on the unit disc $\D$, since the convolution with the Poisson kernel gives the analytic extension of each function from $H[X]$ on $\T$ to the whole $\D$. 
If $1\le p\le\infty$, then $H^p:=H[L^p]$ is the classical Hardy space (see for example \cite{Hof62,Du70,Kat76}).

We shall use also the following variants of Hardy spaces
$$
\overline{H[X]}=\{\bar f\colon  f\in H[X]\}
$$
and 
$$
H_-[X]=\{f\chi_{-1}\colon f\in \overline{H[X]}\}.
$$

Finally, we can introduce main actors of the paper -- Toeplitz and Hankel operators. We will extend the definition (\ref{TH}) to allow possibly large class of symbols, but at this moment we say nothing about boundedness. 
Thus, for a given $a\in L^1$ the Toeplitz $T_a$ operator may be formally defined on 
$\mathcal{P}_A$ (or on $H^{\infty}$) by 
\[
T_a\colon  f\to P(af).
\]
The flip operator $J: L^1\to L^1$ is defined as 
\[
Jf(t)=t^{-1}f(t^{-1}).
\]
Of course, it is  isometry on each r.i. B.f.s. $X$. Consequently, the Hankel operator $H_a$ may be defined on $\mathcal{P}_A$ (or on $H^{\infty}$) by 
\[
H_a\colon f\to P(aJf).
\]
Notice that the definition of Toeplitz operator is rather the same through  literature, while definitions of Hankel operator vary. The definition proposed above corresponds to the one from \cite{BS06} and $H_a$ acts into the space of analytic functions, while, for example in \cite{Pel03}, $H_a$ maps analytic functions into anti-analytic ones. Anyhow, the merit is preserved in any case and  Toeplitz operators have Toeplitz matrices (i.e. $\langle T_a\chi_j,\chi_k\rangle=\hat a(k-j)$ for all $k,j\geq 0$), while Hankel operators are representable by Hankel matrices (i.e. $\langle H_a\chi_j,\chi_k\rangle=\hat a(k+j+1)$ for all $k,j\geq 0$).

We will also consider Toeplitz and Hankel operators on nonseparable spaces. In such a case the above definition of Toeplitz and Hankel operators has to be done more precise, since behavior on polynomials will not determine them. Namely, if $X$ is r.i. B.f.s., then assumption $a\in X'$ ensures that $af, aJf\in L^1$ for each $f\in H[X]$ thus definitions of $T_a$ and $H_a$ make sense. 

\section{Preliminaries}

Before we will be ready to state the main results, we need to collect a sequence of technical results concerning the structure of $H[X]$ spaces. 
Recall that the Fej\'er kernel $(K_n)$ is defined as
\[
K_n(t)=\sum_{k=-n}^n\left(1-\frac{|k|}{n+1}\right)\chi_k(t),
\quad t\in\T.
\]

\begin{lemma}\label{le:dens}
Let $X$ be a r.i. B.f. space. 
If $X$ is separable, then
\begin{enumerate}
\item[(a)]
for every $f\in X$
\[
\lim_{n\to\infty}\|f-f*K_n\|_X=0;
\]

\item[(b)]
$\mathcal{P}$ is dense in $X$;

\item[(c)]
$\mathcal{P}_A$ is dense in $H[X]$.


\end{enumerate}
\end{lemma}
\begin{proof}
(a) It follows from \cite[Chap.~3, Lemma~6.3]{BS88} that continuous functions are dense in $X$, since separability of $X$ means that $X=X_o$.  
Consequently, $X$ is the homogeneous Banach space (in the sense of \cite[Chap. I, Definition 2.10]{Kat76}) and the claim follows by 
\cite[Chap.~I, Theorem~2.11]{Kat76}. Part (b) is an immediate consequence of
part~(a) and the fact that $f* K_n\in\mathcal{P}$ if $f\in X\subset L^1$. Part~(c) follows from part~(a)  and the 
observation that $f*K_n\in\mathcal{P}_A$ if $f\in H[X]$.
\end{proof}

\begin{lemma}\label{le:pol-norm}
Let $X$ be a r.i. B.f.s.. 
Then 
\begin{equation}\label{eq:pol-norm-1}
\|f\|_X=\sup\{|\langle f,p\rangle|\colon p\in\mathcal{P},\ \|p\|_{X'}\le 1\}.
\end{equation}
\end{lemma}
\proof
We know by (\ref{eq:pol-norm-0}) that
\begin{equation}\label{eq:pol-norm-2}
\|f\|_X=\sup\{|\langle f,g\rangle|\colon g\in X',\ \|g\|_{X'}\le 1\}.
\end{equation}
First of all notice that in the above supremum we may restrict to simple functions from $X'$, i.e. 
\begin{equation}\label{eq:pol-norm-3}
\|f\|_X=\sup\{|\langle f,g\rangle|\colon g {\rm \ is\ simple\ function\ and\ } \|g\|_{X'}\le 1\}.
\end{equation}
In fact,  for each $g\in X'$ there is a sequence of simple functions $(g_n)$ such that $|g_n|\leq |g|$ and $g_n\to g$ a.e.. Then  the Lebesgue dominated convergence theorem implies that $\langle f,g_n\rangle\to \langle f,g\rangle$. In particular, if $\|g\|_{X'}\le 1$ then also $\|g_n\|_{X'}\le 1$.

Since $X'$ is r.i. and enjoys the Fatou property, it is an exact interpolation space  between $L^1$ and $L^{\infty}$ (see \cite[Theorem 4.9, p. 105]{KPS82}). In consequence, for each $g\in X'$
\begin{equation}\label{eq:pol-norm-4}
\|g*K_n\|_{X'}\le\|K_n\|_{L^1}\|g\|_{X'}\le\|g\|_{X'},
\quad
n\in\N,
\end{equation}
where $(K_n)$ is the Fej\'er kernel. Moreover, $g*K_n\to g$ a.e. (in fact in each Lebesgue point of $g$, since Fej\'er kernel is approximative unity). However, if we choose $g$ to be simple function then also $|g*K_n|\leq \|g\|_{\infty}\chi_{\T}$. Therefore, using once again the Lebesgue dominated convergence theorem we conclude that $\langle f,g*K_n\rangle\to \langle f,g\rangle$ for each $f\in X$ and each simple function $g\in X'$. Together with (\ref{eq:pol-norm-4}) it proves our claim. \endproof

The idea of the proof of lemma below is analogously as for $H^p$ spaces in \cite{Du70}. We believe it is known, but cannot find any reference. Moreover, it was proved in \cite{K04} with additional assumption, that $X$ is reflexive. To avoid the impresion that this assumption is necessary, we present a short proof. 

\begin{lemma}\label{duality}
Let $X$ be separable r.i. B.f.s. with nontrivial Boyd 
indices. Then $H[X]^*$ is isomorphic with $H[X']$, i.e. $H[X]^*\simeq H[X']$. Moreover, each functional $G\in H[X]^*$  is of the form
$$
G(f)=\langle f,g \rangle=\int_{\T} f(t)\bar g(t)dm(t),
$$
for some unique $g\in H[X']$. Moreover, for such $G$ there holds 
\[
\|G\|_{H[X]^*}\leq \|g\|_{H[X']}\leq \|P\|_{X'\to X'}\|G\|_{H[X]^*}.
\]
\end{lemma}
\proof
Once we know that $H[X]^*\simeq X^*/H[X]^{\perp}$ (since $H[X]$ is closed subspace of $X$)  and $X^*\simeq X'$ (by separability of $X$ and \cite[p. 29]{LT79}), it is enough to prove that $H[X]^{\perp}\simeq H_-[X']$. In fact, since $X'$ has nontrivial Boyd indices when $X$ has, it follows that $P$ is bounded on $X'$, $P(X')=H[X']$ and $H_-[X']$ is complement of $H[X']$ in $X'$. 

Since $X^*\simeq X'$ we may regard elements of  $H[X]^{\perp}$ as functions in $X'$. Let $f\in  H[X]^{\perp}$. Then
\[
\langle \chi_n,f\rangle =0\ {\rm for\ each\ }n\geq 0,
\]
since $\chi_n\in H[X]$. But it means that $f\in H_-[X']$. For the opposite inclusion let $g\in  H_-[X']$. Then for each polynomial $p=\sum_{k=1}^np_k\chi_k\in H[X]$ there holds 
\[
\langle p,g\rangle =\sum_{k=1}^np_k\langle \chi_k,g\rangle=0
\]
and, in view of density of analytic polynomials in $H[X]$, we conclude that $g\in  H[X]^{\perp}$. The remaining inequalities for norms may be explained exactly as in \cite[Section 7.2]{Du70}.
\endproof

\section{Toeplitz operators}

\begin{lemma}\label{le:multiplication}
Let $X,Y$ be r.i. B.f.s. and suppose $X$ is separable.
If a linear operator $A:X\to Y$ is bounded and there exists a sequence $(a_n)_{n\in\N}$ of
complex numbers such that
\begin{equation}\label{eq:multiplication-1}
\langle A\chi_j,\chi_k\rangle=a_{k-j}
\quad\text{for all}\quad
j,k\in\Z,
\end{equation}
then there exists a function $a\in M(X,Y)$ such that $A=M_a$ and \
$\widehat{a}(n)=a_n$ for all $n\in\Z$.
\end{lemma}
\begin{proof}
Put $a:=A\chi_0\in Y$. Since $Y\subset L^1$, we infer from
\eqref{eq:multiplication-1} that
\[
\widehat{a}(n)
=
\langle a,\chi_n\rangle 
=
\langle A\chi_0,\chi_n\rangle 
=
a_n,
\quad
n\in\Z.
\]
If $f=\sum_{k=-m}^m \widehat{f}(k)\chi_k\in\mathcal{P}$, then
$af\in Y\subset L^1$ and the $j$-th Fourier coefficient of $af$ is
\begin{equation}\label{eq:multiplication-2}
(af)\widehat{\hspace{2mm}}(j)
=
\sum_{k\in\Z}\widehat{a}(j-k)\widehat{f}(k)
=
\sum_{k=-m}^m a_{j-k}\widehat{f}(k).
\end{equation}
On the other hand, from \eqref{eq:multiplication-1} we get
for $j\in\Z$,
\begin{equation}\label{eq:multiplication-3}
(Af)\widehat{\hspace{2mm}}(j)
=
\langle Af,\chi_j\rangle
=
\sum_{k=-m}^m \widehat{f}(k)\langle A\chi_k,\chi_j\rangle
=
\sum_{k=-m}^m a_{j-k}\widehat{f}(k).
\end{equation}
By \eqref{eq:multiplication-2} and \eqref{eq:multiplication-3}, 
$(af)\widehat{\hspace{2mm}}(j)=(Af)\widehat{\hspace{2mm}}(j)$ for all
$j\in\Z$. Therefore, $Af=af$ for all $f\in\mathcal{P}$
in view of the uniqueness of the Fourier series. 
Since the space $X$ is separable, the set
$\mathcal{P}$ is dense in $X$ by Lemma~\ref{le:dens}.
In consequence $Af=af$ for all $f\in X$. This means that $A=M_a$
and $a\in M(X,Y)$ by the definition of $M(X,Y)$.
\end{proof}




\begin{theorem}[General Brown--Halmos theorem]\label{Tw-BH}
Let $X,Y$ be two separable r.i. B.f.  spaces, such that 
$X\subset Y$, $Y$ has nontrivial Boyd indices and the Fatou property. A continuous linear operator $A:H[X]\to H[Y]$ satisfies
\begin{equation}\label{Toep1}
\langle A\chi_j,\chi_k\rangle=a_{k-j}
\end{equation}
for some sequence 
$(a_k)_{k\in \mathbb{Z}}$ and all $j,k\geq 0$ if and only if there exists $a\in M(X,Y)$ such that  $A=T_a$ and  $\widehat{a}(n)=a_n$ for all $n\in\Z$.  Moreover,
\begin{equation}\label{Toep1norm}
\|a\|_{M(X,Y)}\leq \|T_a\|_{H[X]\to H[Y]}\leq \|P\|_{Y\to Y}\|a\|_{M(X,Y)}.
\end{equation}
\end{theorem}
\proof
Of course, we need to prove only necessity. 
For $n\geq 0$ put
\[
b_n=\chi_{-n}A\chi_n.
\]
Then $b_n\in Y$ and $\|b_n\|_Y\leq \|A\|_{H[X]\to H[Y]}$. Notice that under 
our assumptions on $Y$, $(Y')_o\not = \{0\}$ and $[(Y')_o]^*=[(Y')_o]'\equiv Y''\equiv Y$ by (\ref{drugi dual}), which means 
that $Y$ is a dual of separable space. In consequence, 
relative weakly* topology of $B(Y)$ is metrizable (see for example  \cite[Corollary 2.6.20]{Meg98} p. 231). 
Thus the Banach-Alaoglu theorem implies that there 
is $a\in Y$, $\|a\|_Y\leq \|A\|_{H[X]\to H[Y]}$ and a sequence $(n_k)$ such that 
$b_{n_k}\to a$  weakly*. In particular, for each $j\in \mathbb{Z}$
\[
\langle b_{n_k},\chi_j\rangle \to \langle a,\chi_j\rangle.
\]
On the other hand,
\[
\langle b_{n_k},\chi_j\rangle=\langle A\chi_{n_k},\chi_{n_k+j}\rangle=a_j,
\]
when $n_k+j\geq 0$. This means that for each $j\in \mathbb{Z}$
\[
\langle a,\chi_j\rangle=a_j.
\]
Consider $B\colon X\to Y$ given by $B\colon f\mapsto af$. Then we have 
\[
\langle Bf,g \rangle=\langle \chi_{-n}A(\chi_nf),g\rangle
\]
for polynomials $f,g\in \mathcal{P}$ and $n>\max\{\deg f,\deg g\}$. Also for these $n$'s there holds 
\[
\|A(\chi_nf)\|_Y
\leq \|A\|_{H[X]\to H[Y]}\|\chi_n f\|_{H[X]}
= 
\|A\|_{H[X]\to H[Y]}\|f\|_{X}.
\]
Thus
\[
|\langle \chi_{-n}A(\chi_nf),g\rangle|\leq \|A\|_{H[X]\to H[Y]}\|f\|_{X}\|g\|_{Y'}
\] 
and
\[
|\langle Bf,g \rangle|\leq 
\limsup_{n\to \infty}|\langle \chi_{-n}A(\chi_nf),g\rangle|
\leq 
\|A\|_{H[X]\to H[Y]}\|f\|_{X}\|g\|_{Y'}.
\]
Taking supremum over $\|f\|_{X}\leq 1,\|g\|_{Y'}\leq 1$, $f,g\in\mathcal{P}$, 
by density of $\mathcal{P}$ in $X$ and by Lemma \ref{le:pol-norm} we conclude 
\[
\|B\|_{X\to Y}\leq \|A\|_{H[X]\to H[Y]}.
\]
Furthermore, for $k,j\in \mathbb{Z}$
\[
\langle B\chi_j,\chi_k \rangle=\langle a,\chi_{k-j} \rangle=a_{k-j}.
\]
Consequently, Lemma \ref{le:multiplication} implies that $a\in M(X,Y)$. On the other hand
\[
\langle T_a\chi_j,\chi_k \rangle=a_{k-j}=\langle A\chi_j,\chi_k\rangle
\]
for $j,k\geq 0$. Moreover, $A\chi_j, T_a\chi_j\in H[Y]\subset H^1$, thus 
\[
A\chi_j=T_a\chi_j
\]
for each $j\geq 0$, by uniqueness of Fourier series. Finally, since $\mathcal{P}_A$ is 
dense in $H[X]$, we conclude that $T_a=A$ and 
\[
\|a\|_{M(X,Y)}
=
\|B\|_{X\to Y}
\leq \|A\|_{H[X]\to H[Y]}=\|T_a\|_{H[X]\to H[Y]},
\] 
as claimed.
\endproof

Indeed, we can slightly relax assumptions from the previous theorem, allowing 
$X$ to be nonseparable. However, then the condition (\ref{Toep1}) no more 
determines an operator, so  Theorem \ref{Tw-BH} rather reads as follows.

\begin{theorem}\label{BHnonsep}
Let $X,Y$ be r.i. B.f. spaces, such that $X\subset Y$,  
$Y$ has nontrivial Boyd indices and the Fatou property. Then the Toeplitz 
operator $T_a:f\mapsto P(af)$ is bounded from $H[X]$ to $H[Y]$ if and only if $a\in M(X,Y)$ and then
\[
\|a\|_{M(X,Y)}\leq \|T_a\|_{H[X]\to H[Y]}\leq \|P\|_{Y\to Y}\|a\|_{M(X,Y)}.
\]
\end{theorem}
\proof
Suppose first $X\not= L^{\infty}$. Then  
$T_a:H[X_o]\to H[Y]$ and has Toeplitz matrix representation, i.e. satisfies  (\ref{Toep1}). Then applying 
Theorem \ref{Tw-BH} (we can, since $X_o$ is separable) we conclude that  $a\in M(X_o,Y)\equiv M(X,Y)$ (see Lemma \ref{order multipliers}). Moreover, respective inequalities are preserved, since 
$\|T_a\|_{H[X_o]\to H[Y]}\leq \|T_a\|_{H[X]\to H[Y]}$. In the case of $X=L^{\infty}$ we cannot use the previous argument, since $X_o=\{0\}$. However, we may take $\cC:=\cC(\mathbb{T})$ instead, which gives disc 
algebra $\cA$ in place of $H[X_o]$. Then the proof of Theorem \ref{Tw-BH} follows 
the same lines, once we know that $M(\cC,Y)\equiv M(L^{\infty},Y)\equiv Y$, but it is 
evident since $\chi_0\in \cC$. 
\endproof

When $X=Y$ we get another corollary of Theorem \ref{Tw-BH}, which improves assumptions of  \cite[Theorem~4.5]{K04}, since we do not require that $X$ is reflexive. 

\begin{corollary}
Let $X$ be a separable r.i. B.f.s. with nontrivial Boyd indices and the Fatou property. If a linear operator $A$ is bounded on $H[X]$ and there exists a sequence  
$(a_n)_{n\in\Z}$ of complex numbers satisfying \eqref{Toep1},
then there exists a function $a\in L^\infty$ such that $A=T_a$ and 
$\widehat{a}(n)=a_n$ for all $n\in\Z$. Moreover,
\[
\|a\|_{L^\infty}
\leq 
\|T_a\|_{H[X]\to H[X]}
\leq 
\|P\|_{X\to X}\|a\|_{L^\infty}.
\]
\end{corollary}

\section{Hankel operators}

In order to prove generalized Nehari theorem we need to state some results on pointwise products of Hardy type spaces. 
The theorem below may be regarded as a kind of regularization for the Lozanovskii's type factorization (see forthcoming paper   \cite{LMM17} for more general treating of this subject). The case of $H[X]\odot H[X']=H^1$ was already considered in  \cite[Theorem 5.2]{K04}. The proof below goes similar lines, but we provide it for the sake of convenience. 

The pointwise product $H[X]\odot H[Y]$ of two Hardy spaces is defined analogously as in (\ref{product}), this is 
\begin{equation*}
H[X]\odot H[Y]=\{ h\in L^0\colon h=fg,f\in H[X],g\in H[Y]\}, 
\end{equation*}
with 
\begin{equation*}
\| h\| _{H[X]\odot H[Y]}=\inf \{ \| f\|_{H[X]}\| g\| _{H[Y]}\colon h=fg\}. 
\end{equation*}
For the moment we do not even know that such a product is a linear space, but it follows at once from the lemma below. 
\begin{theorem}\label{Tw-prod}
Let $X,Y$ be r.i. B.f. spaces with $X\odot Y\subset L^1$. Then
\[
H[X]\odot H[Y]\equiv H[X\odot Y].
\]
\end{theorem}
\proof 
 Suppose that $f\in H[X]$, $g\in H[Y]$, regarded as functions on $\T$. Then, we allow $F,G$ to be extensions of $f$ and $g$ to the unit disc $\mathbb{D}$ by convoluting $f$ and $g$ with the Poisson kernel. Evidently, $F,G$ are analytic, their radial limits exist and are equal a.e. to $f,g$, respectively. In consequence, radial limit of $FG$ is equal a.e. to $fg$ and belongs to $H[X\odot Y]$. Thus $H[X]\odot H[Y]\subset H[X\odot Y]$. 

Let $0\not =h\in H[X\odot Y]$. In particular  $h\in X\odot Y$, thus for each $\epsilon>0$ there are $f\in X, g\in Y$ such that $h=fg$ and $\|f\|_X\|g\|_Y-\epsilon\leq \|h\|_{X\odot Y}\leq \|f\|_X\|g\|_Y$. But $ h\in H[X\odot Y]\subset H^1$ and therefore assumptions of Proposition 5.1 from \cite{K04} are satisfied. Thus 
\[
\tilde F(z)=\exp{\int_{\T}\frac{t+z}{t - z}\log{|f(t)|}dt}
\]
and 
\[
\tilde G(z)=\exp{\int_{\T}\frac{t+z}{t - z}\log{|g(t)|}dt}
\]
are well defined outer functions of $F$ and $G$, respectively (see \cite[Section 2.4]{Du70}). It means that radial limits $\tilde f, \tilde g$ of $\tilde F, \tilde G$ satisfy $|\tilde f|=|f|$ and $|\tilde g|=|g|$.  Letting $\Phi =\frac{H}{\tilde F\tilde G}$ we see that $\Phi$ is analytic ($H$ is extension of $h$ to $\D$), since $\tilde F\tilde G$ have no zeros in $\mathbb{D}$ (in particular, $\Phi$ is an inner function). In consequence,  taking $\phi$ as radial limit of $\Phi$ we see that 
\[
x=\phi \tilde f {\rm\ and\ } y= \tilde g,
\]
which gives the required factorization of $h$, i.e. $h=xy$, $\|x\|_{H[X]}=\|f\|_X$ and $\|y\|_{H[Y]}=\|g\|_Y$.
\endproof

The subsequent lemmas will be used in the proof of Nehari theorem. The second one is rather technical, while the first one is of independent interest and may be regarded as complement of considerations from \cite{KLM14}.

\begin{lemma}\label{cancel}
Let $X,Y$ be two r.i. B.f. spaces with the Fatou property. If $X\odot M(X,Y)=Y$, then 
\[
M(X,Y)'=X\odot Y'.
\]
\end{lemma}
\proof
Of course, the assumption $X\odot M(X,Y)=Y$ implies that $M(X,Y)\not =\{0\}$ and thus $X\subset Y$. 
The Lozanovskii factorization theorem (\ref{Loz fact}) applied twice gives
\begin{equation}\label{skrac1}
M(X,Y)\odot M(X,Y)'\equiv L^1\equiv Y\odot Y'= X\odot M(X,Y) \odot Y'.
\end{equation}
Thus applying Theorem 1 from  \cite{KLM14} we may write 
$$
[M(X,Y)\odot M(X,Y)']^{(4)}= [M(X,Y)^{(2)}]^{1/2}[(M(X,Y)')^{(2)}]^{1/2}.
$$
At the same time we have
\[
[X\odot M(X,Y) \odot Y']^{(4)}=[M(X,Y)^{(2)}]^{1/2}[X^{1/2}(Y')^{1/2}]^{1/2}.
\]
Thus equality (\ref{skrac1}) gives
$$
 [M(X,Y)^{(2)}]^{1/2}[(M(X,Y)')^{(2)}]^{1/2}=[M(X,Y)^{(2)}]^{1/2}[X^{1/2}(Y')^{1/2}]^{1/2}
$$
and applying uniqueness of the Calder\'on--Lozanovskii construction (\cite{CN03} or \cite{BM05}), since all spaces $M(X,Y)^{(2)},(M(X,Y)')^{(2)}$ and $X^{1/2}(Y')^{1/2}$ are B.f. spaces with the Fatou property, we conclude that 
$$
X^{1/2}(Y')^{1/2}=(M(X,Y)')^{(2)},
$$
or, equivalently, 
$$
M(X,Y)'= X\odot Y',
$$
which proves our claim.
\endproof

\begin{lemma}\label{Le-dens}
Let $X,Y$ be two r.i. B.f. spaces such that $X$ is separable and $X\subset Y$. Then the set 
$$
S=\{pq\colon p,q\in 
\mathcal{P}_A {\rm\ and\ } \|p\|_{H[X]}\leq 1, \|q\|_{H[Y']}\leq 1\}
$$ 
is dense in the unit ball of $H[X]\odot H[Y']$.
\end{lemma}
\proof
If $Y=L^{1}$, then $X\odot Y'\equiv X\odot L^{\infty}\equiv X$ and $H[X]\odot H^{\infty}\equiv H[X]$. Since $H[X]$ is separable, the claim follows. 
We can therefore assume that $Y\not = L^1$, which means that $Y'\not =L^{\infty}$ and thus $L^{\infty}\subset (Y')_o$. 

First of all we need to explain that 
\[
X\odot Y'\equiv X\odot (Y')_o. 
\]
In order to do it, we use representation
\[
X\odot Y'\equiv (X^{1/2} (Y')^{1/2})^{(1/2)}
\]
(see \cite[Theorem~1(iv)]{KLM14}). Thus, it is enough to prove that $X^{1/2} (Y')^{1/2}\equiv X^{1/2} (Y')_o^{1/2}$. 
However, both spaces $X^{1/2}(Y')^{1/2}$ and $X^{1/2}[(Y')_o]^{1/2}$ are order continuous, since $X$ is order continuous (see  \cite[Proposition 4]{Re88} or \cite[Theorem 13]{KL10}). Therefore, both have the semi-Fatou property, as order continuous spaces. It follows that their norms are realized by duality as in (\ref{eq:pol-norm-2}) in Lemma \ref{le:pol-norm}. On the other hand, Lozanovskii duality theorem (\ref{lozan dual}), together with the equality (\ref{drugi dual}), tells that their K\"othe duals are both equal $X'^{1/2}Y''^{1/2}$. Thus both spaces have to be equal, because simple functions belong to both of them and are dense there. 
In consequence, also equality $X\odot Y'\equiv X\odot (Y')_o$ is proved. 

Continuing, we see that $X\odot Y'\subset L^1$, because $X\subset Y$.  
Hence, by Theorem~\ref{Tw-prod}, we get 
\[
H[X]\odot H[Y']\equiv H[X]\odot H[(Y')_o]. 
\]
Therefore, it is enough to prove density of $S$ in the unit ball of $H[X]\odot H[(Y')_o]$. 
However, both spaces $X$ and $(Y')_o$
are separable, therefore the set $\mathcal{P}_A\cap B(H[X])$ is dense in 
$B(H[X])$ and  the set $\mathcal{P}_A\cap B(H[(Y')_o])$ 
is dense in $B(H[(Y')_o])$ in view of 
Lemma \ref{le:dens}.

Let $\epsilon>0$ and  $f\in B(H[X]\odot H[(Y')_o])$. Then $(1-\epsilon)f=gh$ for some $g\in H[X]$ and  $h\in H[(Y')_o]$ satisfying $\|g\|_{H[X]}<1$ and  $\|h\|_{H[Y]}<1$. Furthermore, there are $p,q\in \mathcal{P}_A$ such that $\|p\|_{H[X]}<1$, $\|q\|_{H[Y]}<1$ and $\|g-p\|_{H[X]}<\epsilon$, $\|h-q\|_{H[Y]}<\epsilon$. It means that 
\[
\|f- pq\|_{H[X]\odot H[Y']}\leq 2(\epsilon+\|gh- pq\|_{H[X]\odot H[Y']})
\]
\[
\leq 2\epsilon +4\|g\|_{H[X]}\|h-q\|_{H[Y']}+4\|g-p\|_{H[X]}\|q\|_{H[Y']}\leq 8\epsilon,
\]
where the constant $2$ appears when we apply triangle inequality to the quasi norm $\|\cdot \|_{H[X]\odot H[Y']}$ (see \cite[Corollary 1]{KLM14}). 
\endproof

The following lemma is a key for the general Nehari theorem. Let us however postpone its proof to the next part of this section, because we will be able to comment it and its assumptions better, once we know how it works in the proof of Theorem \ref{extTw-Neh}.

\begin{lemma}\label{Ban-env-Cor}
Let $X,Y$ be two r.i. B.f. spaces, such that $X$ is separable, 
$X\subset Y$, $Y$ has nontrivial Boyd indices and one of the following conditions holds:\\
i) $X\odot M(X,Y)=Y$ and $X,Y$ have the Fatou property,\\
ii) $\beta_X<\alpha_Y$.\\
Then for each bounded linear  functional $\phi$ on $\overline{H[X\odot Y']}$ there is  $g\in M(X,Y)$ (not  unique) such that 
\[
\phi(f)=\int_{\T}g(t)f(t)dm(t)
\]
for all $f\in \overline{H[X\odot Y']}$. Furthermore,  
\begin{equation}\label{normrepr}
\|\phi\|_{(\overline{H[X\odot Y']})^*}\approx \dist_{M(X,Y)}(\chi_1g,\overline{H[M(X,Y)]}).
\end{equation}
\end{lemma}

\begin{theorem}[General Nehari theorem]\label{extTw-Neh}
Let $X,Y$ be two r.i. B.f. spaces, such that $X$ is separable, 
$X\subset Y$, $Y$ has nontrivial Boyd indices and one of the following conditions holds:\\
i) $X\odot M(X,Y)=Y$ and $X,Y$ have the Fatou property,\\
ii) $\beta_X<\alpha_Y$.\\
A continuous linear operator $A:H[X]\to H[Y]$ satisfies 
\begin{equation}\label{Hankel condition}
\langle A\chi_j,\chi_k\rangle=a_{k+j+1} {\rm \ for\ all\ } j,k\geq 0
\end{equation}
and some sequence 
$(a_k)_{k>0}$ if and only if there exists $a\in M(X,Y)$ such that $\hat a(n)=a_n$ for each
$n>0$ and $A=H_a$, i.e. $A\colon f\mapsto PaJf$. Moreover, 
\[
c\dist_{M(X,Y)}(a,\overline{ H[M(X,Y)]})
\leq 
\|H_a\|_{H[X]\to H[Y]}
\]
\[
\leq \|P\|_{Y\to Y}\dist_{M(X,Y)}(a,\overline{ H[M(X,Y)]}),
\]
where the constant $c>0$ depends only on spaces $X,Y$. 
\end{theorem}
\proof
If $a\in M(X,Y)$ and $j,k\geq 0$, then 
$$
\langle H_a\chi_j,\chi_k\rangle 
=
\langle PM_a\chi_{-1-j},\chi_k\rangle
=
\hat a(k+j+1),
$$
as required, while $\|H_a\|_{H[X]\to H[Y]}\leq \|P\|_{Y\to Y}\|a\|_{M(X,Y)}$. 
Notice however, that only $\hat a(k)$'s for positive $k$ play in the definition 
of $H_a$, so that we may write 
\begin{align*}
\|H_a\|_{H[X]\to H[Y]}
&\leq 
\|P\|_{Y\to Y}\inf\{\|b\|_{M(X,Y)}\colon \hat b(k)=\hat a(k) {\rm\ for \ each\ } k>0\}
\\
&= 
\|P\|_{Y\to Y}\dist_{M(X,Y)}(a,\overline{ H[M(X,Y)]}),
\end{align*}
as required. 

Let $A$ satisfy condition (\ref{Hankel condition}). We need to find $a\in M(X,Y)$ such that $\hat a(n)=a_n$ for each $n>0$. Define
\[
\gamma(A)=\sup\{\langle A\chi_0, f\rangle\colon  \|f\|_{H[X]\odot H[Y']}\leq 1\}.
\]
Since $S$, as defined in Lemma \ref{Le-dens}, is dense in $B(H[X]\odot H[Y'])$, we have
\begin{equation}\label{fiA}
\gamma(A)=\sup\{
\langle A\chi_0, pq\rangle\colon  \|p\|_{H[X]}\leq 1,  \|q\|_{H[Y']}\leq 1, p,q\in \mathcal{P}_A
\}.
\end{equation}
On the other hand, 
\[
\|A\|_{H[X]\to H[Y]}=
\sup\{\varphi(Ag)\colon  \|g\|_{H[X]}\leq 1,  \|\varphi\|_{H[Y]^*}\leq 1\}
\]
\[
\geq \sup\{\langle Ag, h\rangle\colon  \|g\|_{H[X]}\leq 1,  \|h\|_{H([Y'])}\leq 1\},
\]
since each $h\in H[Y']$ defines functional $\varphi_h(f)=\langle f,h\rangle$ on $H[Y]$ and evidently $\|\varphi_h\|_{H[Y]^*}\leq \|h\|_{H[Y']}$.
In consequence, 
\begin{equation}\label{normA}
\|A\|_{H[X]\to H[Y]}
\geq
\sup\{\langle Ap, q\rangle\colon  \|p\|_{H[X]}\leq 1,  \|q\|_{H([Y'])}\leq 1, p,q\in \mathcal{P}_A\}.
\end{equation}
For a polynomial $p=\sum{p_k}\chi_k$ define $p^c=\sum \overline{p_k}\chi_k$. Then 
\[
p^c(t)= \overline{p(1/t)}
\]
and, since $t\to 1/t$ is the measure preserving transformation on $\T$,
we see that $p\to p^c$ is isometry on the set 
of analytic polynomials in $H[X]$ and $p^{cc}=p$. Therefore, for $f=pq$, $p,q\in \mathcal{P}_A$, simple calculations shows that
\[
\langle A\chi_0, pq\rangle = \langle Ap^c, q\rangle.
\]
Applying the above to formulas (\ref{fiA})  and (\ref{normA}) we get
\begin{equation}\label{han1}
\gamma(A)\leq \|A\|_{H[X]\to H[Y]}.
\end{equation}
Define now the functional $L_A$ on $\overline{H[X]\odot H[Y']}$ by 
\[
L_A(f)=\int_{\T} A\chi_0(t) f(t)dm(t).
\]
Inequality (\ref{han1}) gives $\|L_A\|=\gamma(A)\leq  \|A\|_{H[X]\to H[Y]}$. Therefore, by Lemma \ref{Ban-env-Cor}, there is $d\in M(X,Y)$ such that 
\[
 L_A(\bar f)=\int_{\T} A\chi_0(t)\bar f(t)dm(t)=\int_{\T} d(t)\bar f(t)dm(t),
\]
for each $f\in H[X]\odot H[Y']$. Finally, for each $n\geq 0$ 
\[
a_{n+1}=\langle A\chi_0, \chi_n\rangle=\hat d(n)
\]
and, taking  $a=\chi_1d\in M(X,Y)$, we see that $\hat a(n)=a_n$ for each $n>0$.
Moreover, from (\ref{normrepr}) it follows
\[
\gamma(A)=\|L_A\|\geq c\dist_{M(X,Y)}(a,\overline{ H[M(X,Y)]}),
\]
where $c>0$ depends only on spaces $X,Y$. 
\endproof


To finish the proof it is enough to prove the remaining Lemma \ref{Ban-env-Cor}. In order to do it, we will need some informations on Banach envelopes of quasi-Banach spaces. Given a quasi-Banach space $X$ with separating dual 
one defines a functional on $X$ by 
\[
\|f\|_{X^{\wedge}}=\inf\{\sum_{k=1}^{n}\|f_k\|_X\colon f=\sum_{k=1}^{n}f_k, f_k\in X,n\in \N\}.
\]
Under assumption that $X^*$ separates points of $X$, it is a norm. In this case, the Banach envelope $X^{\wedge}$ of $X$ is defined as the completion of $X$ with respect to the norm $\|\cdot\|_{X^{\wedge}}$. More informations on Banach envelopes may be found in \cite{KPR84, KK16, KC17, KM07, Sh76}. Let us collect some properties of $X^{\wedge}$ for the special kind of spaces $X$ that will appear in our proofs. 
\begin{itemize}
\item[1)] $X$ and $X^{\wedge}$ have the same dual spaces (see \cite[page 27]{KPR84}). 
\item[2)] If $X\subset L^1$ is separable q-B.f.s. then $X^{\wedge}$ is the closure of $X$ in $X''$. 
In fact, assumption $X\subset L^1$ means that $L^{\infty}\subset X'$ and thus $X^*$ separates points of $X$.
Then $X^{\wedge}$ is the closure in $X^{**}$ of the image of $X$ under $J$, where $(Jx)(x^*)=x^*x$ is the natural canonical embedding (see \cite[page 27]{KPR84}). However, by Propositions 3.2 and 3.4 in \cite{KK16} and by separability of $X$ it follows that $X^*=(X^{\wedge})^*\simeq X'$. On the other hand, it is known that the dual $Z^*$ of a B.f.s. may be represented as $Z'\oplus S$, where $S$ is the space of singular functionals (see \cite[Theorem 2, p. 467]{Za67}). Thus $X^{**}\simeq X''\oplus S$. Finally, for separable q-B.f.s. $X$, its image under $J$ is in $X''$, which explains the claim (cf. \cite[p. 232]{KM07}).
\item[3)] If $X\subset L^1$ is a separable r.i. q-B.f.s., then
\begin{equation}
(X'')_o\equiv \overline{X}^{X''},
\end{equation}
where $\overline{X}^{X''}$ denotes the closure of $X$ in $X''$. 
In fact, since $L^{\infty}\subset X$ and $X$ is order continuous, it follows that $X''\not = L^{\infty}$. Therefore $L^{\infty}\subset(X'')_o$ and  $(X'')_o$ is the closure of simple functions in $X''$. On the other hand, simple functions are contained also in $X$. Thus, evidently $(X'')_o\subset \overline{X}^{X''}$. To see that the second inclusion also holds it is enough to notice that $\|f\|_{X''}=\|f\|_{X^{\wedge}}\leq \|f\|_{X}$ for each $f\in X$, which implies that simple functions are dense in $X$ equipped with the norm of $X''$. Of course, $X^{\wedge}$ is also r.i. B.f.s. (cf. \cite[Lemma 2.1]{KM07}).

\end{itemize}

The next lemma gives representation of the Banach envelope of Hardy space $H[Z]$, which will be used in the proof of Lemma \ref{Ban-env-Cor}. It seems to be of independent interest that such a simple representation is possible. Notice that the crucial assumption here is $Z\subset L^1$, which gives that $Z$ has nontrivial dual, in contrast to the situation of $H^p=H[L^p]$, $p<1$, where $(L^p)^*=\{0\}$, but $(H^p)^*\not=\{0\}$ (see \cite[page 115]{Du70}).

\begin{lemma}\label{Ban-env}
Let $Z$ be  a separable r.i. q-B.f.s. such that $Z\subset L^1$ and $Z$ has nontrivial Boyd indices. Then 
\begin{equation}\label{envrep}
H[Z]^{\wedge}=H[Z^{\wedge}].
\end{equation}
\end{lemma}
\proof
First of all notice that for r.i. q-B.f.s. the assumption $Z\subset L^1$ implies $L^{\infty}\subset Z'$. Therefore, $Z'=Z^*$ separates points of $Z$, because  $L^{\infty}$ separates points of $L^1$ and it follows by properties 2) and 3) of Banach envelopes that 
$Z^{\wedge}=(Z'')_o$, where we know that $(Z'')_o\not =\{0\}$, since $Z$ has nontrivial Boyd indices. In particular, $\|f\|_{Z^{\wedge}}= \|f\|_{Z''}$ for $f\in Z^{\wedge}$ and the space $H[Z^{\wedge}]$ is well defined. 

Secondly, $P$ is bounded on  $Z^{\wedge}$, since for each $f\in Z$ there holds 
\[
\|Pf\|_{Z^{\wedge}}\leq \inf\{\sum_{k=1}^{n} \|Pf_k\|_Z\colon f=\sum_{k=1}^{n} f_k, f_k\in Z, n\in \N\}\leq \|P\|_{Z\to Z}\|f\|_{Z^{\wedge}}
\]
and $Z$ is dense in $Z^{\wedge}$. 
Then it follows that $H[Z]$ is dense in $H[Z^{\wedge}]$, because $Z$ is dense in $Z^{\wedge}$ and $P$ is bounded on both $Z$ and $Z^{\wedge}$. Moreover, $H[Z]$ is dense in $H[Z]^{\wedge}$ just by definition. 

Now let $f\in H[Z]$. Then its norm in $H[Z^{\wedge}]$ is  given by 
\[
\|f\|_{H[Z^{\wedge}]}=\|f\|_{Z^{\wedge}}=\inf\{\sum_{k=1}^{n} \|f_k\|_Z\colon  f=\sum_{k=1}^{n} f_k, f_k\in Z, n\in \N\}.
\]
On the other hand, the norm of the same $f$ regarded as an element of $H[Z]^{\wedge}$ is
\[
\|f\|_{H[Z]^{\wedge}}=\inf\{\sum_{k=1}^{n} \|f_k\|_Z\colon f=\sum_{k=1}^{n} f_k, f_k\in H[Z], n\in \N\}.
\]
Evidently, $\|f\|_{H[Z^{\wedge}]}\leq \|f\|_{H[Z]^{\wedge}}$, while the opposite inequality follows from boundedness of $P$ on $Z$. This is
\[
\|f\|_{H[Z^{\wedge}]}=\|f\|_{Z^{\wedge}}=\inf\{\sum_{k=1}^{n} \|f_k\|_Z\colon  f=\sum_{k=1}^{n} f_k, f_k\in Z, n\in \N\}
\]
\[
\geq 1/\|P\|_{Z\to Z} \inf\{\sum_{k=1}^{n} \|Pf_k\|_Z\colon  f=\sum_{k=1}^{n} f_k, f_k\in Z, n\in \N\}
\]
\[
\geq 1/\|P\|_{Z\to Z} \inf\{\sum_{k=1}^{n} \|g_k\|_Z\colon  f=\sum_{k=1}^{n} g_k, g_k\in H[Z], n\in \N\}
\]
\[
= 1/\|P\|_{Z\to Z}\|f\|_{H[Z]^{\wedge}}.
\]
Consequently, $H[Z^{\wedge}]=H[Z]^{\wedge}$ as completions of the same space under equivalent norms. 
\endproof

\proof[Proof of Lemma \ref{Ban-env-Cor}]
Assume that the condition i) is satisfied. Lemma \ref{cancel} gives 
$$
M(X,Y)'= X\odot Y'.
$$
Consequently, $M(X,Y)\equiv M(X,Y)''= (X\odot Y')'$ and, since  $X\odot Y'$ is separable (see the proof of Lemma \ref{Le-dens}), we get
$$
(X\odot Y')^*=M(X,Y).
$$
Furthermore, $\overline{H[X\odot Y']}$ is the closed subspace of the Banach space $X\odot Y'$, thus the claim follows by the Hahn-Banach theorem.

Suppose now that condition ii) holds. This time $X\odot Y'$ need not be a Banach space, therefore we will use Lemma \ref{Ban-env}. In order to do it we need to see that $X\odot Y'$ has nontrivial Boyd indices. 
First of all, it is easy to see that  
\begin{equation}\label{indeksyZ}
\alpha_{Z^{(1/2)}}=2\alpha_{Z} {\rm \ and\ } \beta_{Z^{(1/2)}}=2\beta_{Z}, 
\end{equation}
for an arbitrary r.i. q-B.f.s. $Z$. Furthermore, 
 $\beta_{X^{1/2}(Y')^{1/2}}\leq 1/2\beta_{X}+1/2\beta_{Y'}$ and together with the representation $X\odot Y'\equiv (X^{1/2} (Y')^{1/2})^{(1/2)}$ we get 
\[
\beta_{X\odot Y'}\leq \beta_X+\beta_{Y'}=1+\beta_X-\alpha_{Y}<1.
\]
Using once again equalities (\ref{indeksyZ}) and representation of $X\odot Y'$ we see that $\alpha_{X\odot Y'}>0$ if and only if $\alpha_{X^{1/2}Y'^{1/2}}>0$, which, in turn, is equivalent to $\beta_{[X^{1/2}Y'^{1/2}]'}<1$ (notice that $X^{1/2}Y'^{1/2}$ is order continuous, thus satisfies the semi-Fatou property and we are free to use (\ref{indeksysuma})). Using Lozanvskii duality theorem and assumption that $Y$, and so $Y''$, has nontrivial Boyd indices we conclude 
\[
\beta_{[X^{1/2}Y'^{1/2}]'}=\beta_{X'^{1/2}Y''^{1/2}}\leq \frac{1}{2} \beta_{X'}+\frac{1}{2} \beta_{Y''}<1,
\]
which proves that $X\odot Y'$ has nontrivial Boyd indices.

Applying now Lemma  \ref{Ban-env} to $Z=X\odot Y'$ we get $\overline{H[X\odot Y']}^{\wedge}=\overline{H[X\odot Y']^{\wedge}}=\overline{H[(X\odot Y')^{\wedge}]}$. It follows, by properties of Banach envelope that
\[
(\overline{H[X\odot Y']})^*=(\overline{H[X\odot Y']}^{\wedge})^*=(\overline{H[(X\odot Y')^{\wedge}]})^*.
\]
On the other hand
$$
((X\odot Y')^{\wedge})^*=(X\odot Y')^*=(X\odot Y')'=M(X,Y).
$$
Consequently, each functional on $\overline{H[X\odot Y']}$ extends to a functional on $(X\odot Y')^{\wedge}$, which proves the claim. 

It remains to explain (\ref{normrepr}). We will do it only for the case ii), since it works similarly, but easier for i). Let $\phi\in  (\overline{H[X\odot Y']})^*$. Then 
\[
\|\phi\|_{(\overline{H[X\odot Y']})^*}\approx \|\phi\|_{(\overline{H[(X\odot Y')^{\wedge}]})^*}
=\inf\|\tilde\phi\|_{((X\odot Y')^{\wedge})^*},
\]
where the infimum runs over all extensions $\tilde\phi$ of $\phi$ to  $(X\odot Y')^{\wedge}$. Of course, each such extension corresponds to some $\tilde g\in M(X,Y)$, thus we get
\[
\|\phi\|_{(\overline{H[X\odot Y']})^*}\approx \dist_{M(X,Y)}(g,H_-[M(X,Y)])=\dist_{M(X,Y)}(\chi_1g,\overline{H[M(X,Y)]}).
\] 
Notice that we have lost equality of norms because of  Lemma \ref{Ban-env}.
\endproof

\begin{remark}
The second part of the above proof could be done without Lemma  \ref{Ban-env}. In fact, we have explained that under assumption ii), $P$ is bounded on $X\odot Y'$. Consequently, if $\phi$ is continuous functional on $\overline{H[X\odot Y']}$, then 
\[
\phi\circ P:X\odot Y'\to \C
\]
defines its extension to the whole $X\odot Y'$. Since $(X\odot Y')^*=M(X,Y)$, we see that there is $f\in M(X,Y)$ that represents $\phi\circ P$, as well as its restriction $\phi$ to $\overline{H[X\odot Y']}$. We are therefore obliged to explain our choice of argument. First of all this argument does not imply directly the lower estimate of the norm of Hankel operator. 
On the other hand the author believes that Lemma  \ref{Ban-env} holds without assumption on Boyd indices of $Z$, or, at least,  boundedness of $P$ on $Z$ is not necessary. Thus, if we can prove (\ref{envrep}) for some  space $Z=X\odot Y'$ without boundedness of $P$, then the general Nehari theorem holds for $X,Y$ with the same proof as in Theorem \ref{extTw-Neh}.
\end{remark}
\begin{question}
Is it true that $H[Z]^{\wedge}=H[Z^{\wedge}]$, for each separable r.i. q-B.f.s.  $Z$ with $Z\subset L^1$?
\end{question}

It was already known since the paper of Janson,  Peetre and Semmes \cite{JPS84} (see also the classical paper \cite{CRW76} where usefulness of weak factorization in harmonic analysis was exhibited for the first time) that strong factorization from the original proof of Nahari theorem may be replaced by weak factorization, i.e. instead of factorization $f=gh$, we have only  $f=\sum_{n=1}^{\infty} g_nh_n$ (this idea was also used in \cite{TV87} and \cite{BG10}). 
It is worth to notice that in our argument with Banach envelope of Hardy spaces the  weak factorization  is hidden as well. Namely, we have that each $f\in \overline{H[X\odot Y']}^{\wedge}=\overline{H[X]\odot H[ Y']}^{\wedge}$, by Theorem \ref{Tw-prod} and properties of Banach envelopes, admits weak factorization of the form 
\[
f=\sum_{n=1}^{\infty} g_nh_n,
\]
where $g_n\in \overline{H[X]}$, $h_n\in \overline{H[Y']}$ and 
$$
\|f\|_{\overline{H[X\odot Y']}^{\wedge}}\approx \sum_{n=1}^{\infty} \|g_n\|_{\overline{H[X]}}\|h_n\|_{\overline{H[Y']}}.
$$

Similarly, as in the previous section, we may remove assumption on separability of $X$ from Theorem \ref{extTw-Neh}, but then it has a slightly different form.

\begin{theorem}\label{nonsepNeh}
Let $X,Y$ be two r.i. B.f. spaces, such that 
$X\subset Y$, $Y\in (FP)$, $Y$ has nontrivial Boyd indices and one of the following conditions holds:\\
i) $X\odot Y'=X_o\odot Y'$ is a Banach space,\\ 
ii) $\beta_X<\alpha_Y$ and  $X\not=L^{\infty}$,\\
iii) $X=L^{\infty}$ and $Y'$ is separable.\\
If the operator $H_b=PM_bJ\colon H[X]\to H[Y]$ is bounded, 
then there exists $a\in M(X,Y)$ such that $\hat a(n)=\hat b(n)$ for 
$n>0$ and $H_b=H_a$. Moreover, 
\[
c\dist_{M(X,Y)}(a,\overline{ HM(X,Y)})
\leq 
\|H_a\|_{H[X]\to H[Y]}
\leq \|P\|_{Y\to Y}\dist_{M(X,Y)}(a,\overline{ H[M(X,Y)]}),
\]
where the constant $c$ depends only on spaces $X,Y$. 
\end{theorem}
\proof It is, of course, enough to consider only the case of nonseparable $X$. Suppose firstly, that $X\not= L^{\infty}$ and $X\odot Y'=X_o\odot Y'$ is a Banach space or $\beta_X<\alpha_Y$. Then $X_o\not =\{0\}$ and the thesis of Lemma \ref{Ban-env-Cor} follows for $X_o$ in place of $X$. Consequently, we just apply Theorem \ref{extTw-Neh} with $X_o$ and $Y$ in place of $X,Y$. 

In the case of $X= L^{\infty}$  we take disc algebra $\cA$ in place of $H[X_o]$. It is evident that $ L^{\infty}\odot Y'=\cC\odot Y'=Y'$ as well as $ H^{\infty}\odot H[Y']=\cA\odot H[Y']=H[Y']$. Once we know that $Y'$ is separable, Lemma \ref{Le-dens} holds and we may follow the proof of Theorem \ref{extTw-Neh} with  $\cA$ and $\cC$ in place of $H[X]$ and $X$, respectively. 
\endproof

Recall that Lorentz space $L^{p,q}$, where $0< q\leq \infty$ and $0<p< \infty$, is defined by the (quasi-) norm
\[
\|f\|_{L^{p,q}}=(\int_{0}^1 [f^*(s)s^{1/p}]^q\frac{ds}{s})^{1/q},
\]
with the standard modification when $q=\infty$.

A function $\varphi\colon [0,\infty)\to [0,\infty]$ is called the Orlicz function when it is convex, nondecreasing and $\varphi(0)=0$. Then the Orlicz space $L^{\varphi}$ is defined by the norm 
\[
\|f\|_{\varphi}=\inf\{\lambda>0\colon \int_{\T}\varphi(|f(t)|/\lambda)dm(t)\leq 1\}.
\]
We will use the standard notion $H^{p,q}:=H[L^{p,q}]$ and $H^{\varphi}=H[L^{\varphi}]$ for Hardy--Lorentz and Hardy--Orlicz spaces, respectively. 

Let us also mention, that description of the space of multipliers between two Orlicz spaces was already described in full generality in \cite{LT17}. This reads as follows
\[
M(L^{\varphi_1},L^{\varphi})=L^{\varphi\ominus\varphi_1},
\]
where the generalized Legendre transform $\varphi\ominus\varphi_1$ is defined as
\[
\varphi\ominus\varphi_1(u)=\sup_{v>0}\{\varphi(uv)-\varphi(v)\}\ {\rm for\ } u\geq0.
\]
It was also proved therein that $L^{\varphi_1}$ factorizes $L^{\varphi}$ if and only if there are constants $c,C,u_0>0$ such that for each $u>u_0$ there holds
\begin{equation}\label{Orfac}
c \varphi_1^{-1}(u)(\varphi\ominus\varphi_1)^{-1}(u)\leq \varphi^{-1}(u)\leq C \varphi_1^{-1}(u)(\varphi\ominus\varphi_1)^{-1}(u),
\end{equation}
where $\varphi^{-1}$ stands for the right continuous inverse of $\varphi$. 

Recall finally that Boyd indices of the Orlicz space  $L^{\varphi}$ coincides with Matuszewska--Orlicz indices of the function $\varphi$, i.e. 
\[
\alpha_{L^{\varphi}}=a_{\varphi}\ {\rm  and\ } \beta_{L^{\varphi}}=b_{\varphi},
\] 
where $a_{\varphi}$ is the lower and $b_{\varphi}$ - the upper Matuszewska--Orlicz index of $\varphi$ (we refer to \cite{LT79, Mal89} for respective definitions).

\begin{theorem}\label{Orlicz}
Let $\varphi,\varphi_1$ be two Orlicz functions such that $\varphi$ has nontrivial Matu\-szew\-ska--Orlicz indices.\\
a) The Toeplitz operator $T_a=PM_a$ is bounded from $H^{\varphi_1}$ to $H^{\varphi}$ if and only if $a\in L^{\varphi\ominus \varphi_1}$ and 
\[
\|a\|_{L^{\varphi\ominus \varphi_1}}\leq \|T_a\|_{H^{\varphi_1}\to H^{\varphi}}\leq \|P\|_{Y\to Y}\|a\|_{L^{\varphi\ominus \varphi_1}}.
\]
b) If additionally one of the conditions holds:\\
i)  inequalities (\ref{Orfac}) are satisfied and  $L^{\varphi_1}\not =L^{\infty}$,\\
ii) $b_{\varphi_1}<a_{\varphi}$ and  $L^{\varphi_1}\not =L^{\infty}$,\\ 
iii)  $L^{\varphi_1} =L^{\infty}$ and $L^{\varphi}$ is separable,\\
then the Hankel operator $H_b=PM_bJ$ is bounded from $H^{\varphi_1}$ to $H^{\varphi}$ if and only if  there exists $a\in L^{\varphi\ominus \varphi_1}$ such that $\hat a(n)=\hat b(n)$ for each
$n>0$. In this case $H_a=H_b$ and 
\[
c\dist_{ L^{\varphi\ominus \varphi_1}}(a,\overline{  H^{\varphi\ominus \varphi_1}})
\leq 
\|H_a\|_{H^{ \varphi_1}\to H^{ \varphi}}
\leq \|P\|_{L^{\varphi}\to L^{\varphi}}\dist_{ L^{\varphi\ominus \varphi_1}}(a,\overline{  H^{\varphi\ominus \varphi_1}}).
\]
\end{theorem}
\proof The proof is a routine verification of assumptions of Theorem \ref{nonsepNeh} together with \cite[Theorem 2]{LT17}.
\endproof

\begin{theorem}\label{Lorentz}
Let $1\leq p_1, q_1, q\leq \infty$ and $1<p<\infty$. Assume also that either $p< p_1$, or $p=p_1$ and $q>q_1$.  Put  $p_2=\frac{pp_1}{p_1-p}$ when $p<p_1$, or $p_2=\infty$ when $p=p_1$ and $q_2=\frac{qq_1}{q_1-q}$ when $q< q_1$, or $q_2=\infty$ when $q\geq q_1$.\\
a) The Toeplitz operator $T_a=PM_a$ is bounded from $H^{p_1,q_1}$ to $H^{p,q}$ if and only if $a\in L^{p_2,q_2}$ and 
\[
\|a\|_{L^{p_2,q_2}}\leq \|T_a\|_{H^{p_1,q_1}\to H^{p,q}}\leq \|P\|_{L^{p,q}\to L^{p,q}}\|a\|_{L^{p_2,q_2}}.
\]
b) If $p<p_1$ or $p=p_1$, $q=q_1$, then 
then the Hankel operator $H_b=PM_bJ$ is bounded from  $H^{p_1,q_1}$ to $H^{p,q}$  if and only if  there exists $a\in L^{p_2,q_2}$ such that $\hat a(n)=\hat b(n)$ for each
$n>0$. In this case $H_a=H_b$ and 
\[
c\dist_{L^{p_2,q_2}}(a,\overline{  H^{p_2,q_2}})
\leq 
\|H_a\|_{H^{p_1,q_1}\to H^{p,q}}
\leq \|P\|_{L^{p,q}\to L^{p,q}}\dist_{L^{p_2,q_2}}(a,\overline{  H^{p_2,q_2}}).
\]
\end{theorem}
\proof The condition $p< p_1$ or $p=p_1$ and $q>q_1$ guarantees that $L^{p_1,q_1}\subset L^{p,q}$ and $1<p<\infty$ means that $P$ is bounded on $L^{p,q}$. It is also known that in such a case $M(L^{p_1,q_1},L^{p,q})=L^{p_2,q_2}$  (see \cite[Theorem 4]{KLM18}). Therefore, application of Theorem \ref{BHnonsep} proves the case of Toeplitz operators. 
It follows that the assumption $p< p_1$ implies point ii) of Theorem \ref{nonsepNeh}, while the case of $p=p_1$, $q=q_1$ is evident. Thus point b) is also proved. 
\endproof

The following example illustrates the independence of assumptions i) and ii) in Theorem \ref{extTw-Neh}.

\begin{example}
Consider $X=L^{p_1,1}$, $Y=L^{p,\infty}$ with  $1<p<p_1<\infty$. Then $X$ is separable, $P$ is bounded on $Y$ and $X\subset Y$, as required in Theorem \ref{extTw-Neh}. Furthermore, from \cite{KLM14} Theorem  10 we know that $X$ does not factorize $Y$, i.e. assumption i) is not satisfied. On the other hand, $Y'=L^{p',1}$ and applying   \cite[Theorem 7]{KLM14}, we have  
\[
X\odot Y'=L^{p_1,1}\odot L^{p',1}=L^{p_2,1/2},
\]
where $1/p_2=1/p_1+1/p'=1+1/p_1-1/p<1$. This means that $L^{p_2,1/2}$ satisfies condition ii) in Theorem \ref{extTw-Neh}.

On the other hand, taking a separable r.i. B.f.s. $Y$ with nontrivial Boyd indices and $X\equiv Y$, condition i) of Theorem \ref{extTw-Neh} is trivially satisfied, since we have by Lozanovskii factorization theorem $X\odot X'\equiv X_o\odot X'\equiv L^1$. Evidently, condition ii) cannot be satisfied in this case. 
\end{example}

\section{Measure of noncompactness of Toeplitz operators}
It is well known (see \cite{BH63/64,BS06}) that Toeplitz operator $T_a$ acting on $H^2$ is never compact, unless the trivial case of $a=0$. It follows from the fact that $(\chi_n)_{n=0}^{\infty}$ is weakly null in $H^2$ while $T_a\chi_n\not \to 0$ in $H^2$. This proof could be generalized to the case of $T_a:H[X]\to H[Y]$ when $P$ is bounded on $Y$ and $(\chi_n)_{n=0}^{\infty}$ is weakly null sequence in $H[X]$ (it happens when characteristic functions are order continuous in $X$ - see \cite{MRP15}). However,  $(\chi_n)_{n=0}^{\infty}$ is not weakly null for example  in $H^{\infty}$ (neither in the disc algebra $\mathcal{A}$) and therefore we need a new argument that applies in the general situation of our considerations. It apears that our argument  gives estimation of the measure of noncompactness of Toeplitz operator. 

Let us recall that for a given set $A$ in a Banach space, its measure of noncompactness $\alpha(A)$ is defined as 
\[
\alpha(A)=\inf\{\delta>0\colon A\subset \sum_{k=1}^{N}B_k, diam(B_k)\leq \delta \ {\rm and\ } N<\infty\}.
\]
Then the measure of noncompactness $\alpha(T)$ of a given operator $T:X\to Y$ is just defined as the measure of noncompactness of the set $T(B(X))$ in $Y$, where $B(X)$ is the unit ball of $X$. 

\begin{theorem}\label{comp-Toepl}
Let $X,Y$ be r.i. B.f. spaces such that $X\subset Y$ and $Y$ has nontrivial Boyd indices. Suppose $a\in M(X,Y)$. Then 
the measure of noncompactness of the Toeplitz operator $T_a$ satisfies 
$$
\alpha(T_a)\geq m\max_{n\in \Z}|\hat a(n)|,
$$
where $m$ is a constant of inclusion $Y\subset L^1$. In particular, 
operator $T_a:H[X]\to H[Y]$ is compact if and only if $a=0$.
\end{theorem}

\proof
Assume that $a\not =0$. 
We will show that for each $\epsilon>0$ there is a sequence of indices $(k_n)_{n=0}^{\infty}$  such that $\|T_a\chi_{k_n}- T_a\chi_{k_l}\|_{H[Y]}\geq  (1-\epsilon)m\|(\hat a(n))\|_{\infty}$ for each $0\leq n,l$ with $n\not = l$. Of course, it will imply our claim. 

First of all notice that $a\in M(X,Y)$ implies that $a\in Y\subset L^1$.  Since $a\chi_k\in Y$, for each $k$, and $P$ is bounded on $Y$ we have $T_a\chi_k=P(a\chi_k)\in Y$. In consequence $T_a\chi_k\in L^1$ for each $k$. 

Let $c=\sup_{n\in \Z}|\hat a(n)|$. Then, by the Riemann--Lebesgue theorem, there is $s\in \Z$ such that $c=|\hat a(s)|$. We put $k_0=-\min\{0,s\}$. 
Notice that $\widehat{T_a\chi_k}(n)=\hat a(n-k)$ for $n\geq 0$ and $\widehat{T_a\chi_k}(n)=0$ for $n< 0$. In particular, for each $k\geq k_0$ we have $\widehat{T_a\chi_k}(s+k)=\hat a(s)$ and for these $k$ there holds 
\[
\|T_a\chi_k\|_1\geq c=|\hat a(s)|.
\]
We are in a position to find the announced sequence. We have already determined $k_0$. Without lost of generality we may assume that $s<0$, i.e. $k_0=-s$ (if $k_0=0$ the proof is analogous). Fix $\epsilon>0$.
 Thanks to the Riemann--Lebesgue theorem  there is $k_1>k_0$ such that for each  $k\geq k_1$ there holds $|\hat a(-k)|\leq c\epsilon$. Then 
\[
\|T_a\chi_{k_0}-T_a\chi_{k_1}\|_1\geq |\widehat{T_a\chi_{k_0}}(0)-\widehat{T_a\chi_{k_1}}(0)|=|\hat a(-k_0)-\hat a(-k_1)|\geq (1-\epsilon)c.
\] 
Then we choose $k_n=k_0+n(k_1-k_0)$. Thus for $0\leq d<n$ there holds
\[
\|T_a\chi_{k_d}-T_a\chi_{k_n}\|_1\geq |\widehat{T_a\chi_{k_d}}(d(k_1-k_0))-\widehat{T_a\chi_{k_n}}(d(k_1-k_0))|
\]
\[
=|\hat a(-k_0)-\hat a(-k_0-(n-d)(k_1-k_0))|\geq (1-\epsilon)c,
\] 
where the last inequality follows from the fact that $k_0+(n-d)(k_1-k_0)\geq k_1$, which in turn implies $|\hat a(-k_0-(n-d)(k_1-k_0))|\leq \epsilon$. 
Finally, since $\|f\|_Y\geq m\|f\|_1$ for each $f\in Y$ we conclude that 
$\|T_a\chi_{k_n}- T_a\chi_{k_l}\|_{H[Y]}\geq (1-\epsilon)mc$ for each $0\leq n,l$ with $n\not = l$. 
\endproof

{\bf Acknowledgments} The author wish to thank Professor Alexei Karlovich for inspiring discussions on the subject and  Professors Lech Maligranda and Anna Kami\'nska for valuable remarks and advices. He is also grateful to the anonymous referee  for many valuable remarks that significantly improved  presentation of the paper.


\end{document}